\documentclass[10pt]{amsart}
\usepackage{amsmath,amsthm,amssymb,amsfonts, eucal, amscd, mathbbol, mathrsfs} 
\usepackage{pdfsync}
\usepackage{url}
\usepackage{cite} 
\usepackage{setspace}
\usepackage[all,cmtip]{xy} 
\usepackage{graphicx}
\usepackage{subfig}  
\usepackage{fancyhdr}   
\usepackage{latexsym} 
\usepackage{epic}  
\usepackage{ifthen}
\usepackage[bookmarks,bookmarksnumbered, plainpages=false, pdfpagelabels]{hyperref}
\usepackage{rotating}
\usepackage{scalefnt} 
\newtheorem{thm}{Theorem}
\newtheorem{cor}[thm]{Corollary} 
\newtheorem{lem}[thm]{Lemma} 
\newtheorem{prop}[thm]{Proposition}
\theoremstyle{definition} 

\theoremstyle{definition} 

\theoremstyle{definition} 

\theoremstyle{definition}
\newtheorem{remarks}[thm]{Remarks}
\theoremstyle{definition}

\theoremstyle{definition} 

\theoremstyle{definition} 
\newtheorem{example}[thm]{Example}

\numberwithin{thm}{subsection}

\newcommand{\R}{\ensuremath{\mathbb{R}}}

\usepackage{eucal} 

\usepackage{mathrsfs} 

\def\i{\infty}
\def\a{\alpha}

\def\b{\beta}

\def\L{\Lambda} 
\def\o{\omega} 
 
\def\s{\sigma} 
\def\t{\tau}

\def\B{\mathcal{B}}
\def\hB{\hat{\mathcal{B}}}

\def\cal{\mathcal}

\def\rotateminus{\reflectbox{\rotatebox[origin=c]{155}{\hspace{.6pt}-}}}

\def\Xint#1{\mathchoice
{\XXint\displaystyle\textstyle{#1}}%
{\XXint\textstyle\scriptstyle{#1}}%
{\XXint\scriptstyle\scriptscriptstyle{#1}}%
{\XXint\scriptscriptstyle\scriptscriptstyle{#1}}%
\!\int}
\def\XXint#1#2#3{{\setbox0=\hbox{$#1{#2#3}{\int}$}
\vcenter{\hbox{$#2#3$}}\kern-.5\wd0}}

\def\cint{\Xint    \rotateminus }

\begin{document} 
 \author[J. Harrison \& H. Pugh]{J. Harrison \\Department of Mathematics \\University of California, Berkeley\\ \\ H. Pugh \\ Department of Pure Mathematics and Mathematical Statistics\\University of Cambridge} 
	\title[Topology of Chains]{ Topological Aspects of Differential Chains }   
  \today

\begin{abstract}
	In this paper we investigate the topological properties of the space of differential chains \(   \,^\prime {\cal B}(U)  \) defined on an open subset \( U \) of a Riemannian manifold \( M \). We show that  \(  \,^\prime {\cal B}(U)  \) is not generally reflexive, identifying a fundamental difference between currents and differential chains.  We also give several new brief (though non-constructive) definitions of the space $ \,^\prime {\cal B}(U) $, and prove that it   is a separable ultrabornological \( (DF) \)-space.
	
		  Differential chains are closed under dual versions of fundamental operators of the Cartan calculus on differential forms \cite{OC}  \cite{poincarelemma}.   The space has good properties some of which are not exhibited by currents \( {\cal B}'(U) \mbox{ or } {\cal D}'(U) \).   For example,  chains supported in finitely many points are dense in \(  \,^\prime {\cal B}(U) \) for all open \( U \subset M \), but not generally in the strong dual topology of \( {\cal B}'(U) \).
		
		 \end{abstract}				   

  \maketitle  
\section{Introduction} 
\label{sec:introduction} 
 
   We begin with a Riemannian manifold \( M \).  Let \( U \subset M \) be open and    \( {\cal P}_k = {\cal P}_k(U) \) the space of finitely supported sections of the $k$-th exterior power of the tangent bundle \( \L_k(TU) \).    Elements of \(  {\cal P}_k(U) \) are called \emph{pointed $k$-chains in \( U \)}.
Let \( {\cal F} = {\cal F}(U) \) be a complete locally convex  space   of differential forms defined on \( U \).     We find a   predual to  \( {\cal F} \), that is, a complete l.c.s.   \( \,^\prime{\cal F} \) such that \(( \,^\prime{\cal F})' =  {\cal F} \).  The predual \( \,^\prime{\cal F} \) is uniquely determined if we require that \( {\cal P}_k \) be dense in \(  \,^\prime{\cal F}  \), and that the topology on \( \,^\prime{\cal F} \) restricts to the Mackey topology on \( {\cal P}_k \),  the finest locally convex topology on \( {\cal P}_k \) such that  \(( \,^\prime{\cal F})' =  {\cal F} \).  Two natural questions arise:  
(i) Is \(  \,^\prime{\cal F} \) reflexive? 
(ii) Is there a constructive definition of the topology on \( \,^\prime{\cal F} \)?  

   Let \( {\cal E}_k \) be the space of \( C^\i \) $k$-forms defined on \( U \), and \( {\cal D}_k \)  the space of $k$-forms with compact support in \( U \).  Then \( {\cal D}_k \) is an \( (LF) \)-space, an inductive limit of Fr\'echet spaces.   The space \( {\cal D}_k' \)  is the celebrated space of Schwartz distributions for \( k = 0 \) and \( U = \R^n \). Let  \( {\cal B}_k^r \)  be the Fr\'echet space of $k$-forms whose Lie derivatives are bounded up to order \( r \), and  \( {\cal B}_k = \varprojlim_{r} {\cal B}_k^r \).  

  Most of this paper concerns the space \( \,^\prime{{\cal B}_k}  \) which is now well developed (\cite{harrison1, continuity, hodge,  harrison2, OC,   poincarelemma}).   First appearing in \cite{harrison1} is a constructive, geometric definition using ``difference chains,'' which does not rely on a space of differential forms.  (See also \cite{thesis}   for an elegant exposition.   An earlier approach based on polyhedral chains can be found in \cite{diffcomp}.) We use an equivalent definition below using differential forms \( {\cal B}_k \) and the space of pointed $k$-chains \( {\cal P}_k \) which has the advantage of of brevity.  

  We can write an element \( A \in {\cal P}_k(U) \) as a  formal sum \( A = \sum_{i=1}^s (p_i; \a_i) \) where \( p_i \in U \), and \( \a_i \in \L_k(T_p U) \).  Define a family of norms on \( {\cal P}_k \),  \[ \|A\|_{B^r} = \sup_{0 \ne \o \in {\cal B}_k^r} \frac{\cint_A \o}{\|\o\|_{C^r}} \]  for \( r \ge 0 \), where \( \cint_A \o := \sum_{i=1}^s \o(p_i; \a_i) \).   Let \( \hat{\cal{B}}_k^r \) denote the Banach space on completion, and \( \hat{\cal{B}}_k = \varinjlim_r \hat{\cal{B}}_k^r \) the inductive limit.   We endow \( \hat{\cal{B}}_k \) with the inductive limit topology \( \t \). Since the norms are decreasing, the Banach spaces form an increasing nested sequence.  	As of this writing, it is unknown whether  \(  \hat{\cal{B}}_k \) is complete.  Since \(  \hat{\cal{B}}_k \) is a locally convex space, we may take its completion (see Schaefer \cite{scheafer}, p. 17) in any case, and denote the resulting space by \(  \,^\prime{\cal B}_k(U) \). Elements of \(  \,^\prime{\cal{B}}_k(U) \) are called ``differential $k$-chains\footnote{previously known as ``\( k \)-chainlets''} in \( U \).''

The reader might ask   how  \( \,^\prime{\cal B}_k(U) \) relates to the space \(  {\cal{B}}_k'(U) \)  of currents, endowed with the strong dual topology.  We prove below that \( \,^\prime{\cal B}_k (U) \) is not generally reflexive. However, under the canonical inclusion \( u:  \,^\prime{\cal B}_k (U) \to {\cal B}_k'(U) \), this subspace of currents is closed under the primitive and fundamental operators used in the Cartan calculus (see  Harrison \cite{OC}).  Thus, differential chains form a distinguished subspace of currents that is constructively defined and approximable by pointed chains.  That is, while \( {\cal P}_k \) is dense in \(  {\cal{B}}_k' \) in the weak topology, \( {\cal P}_k \)  is in fact dense in \( \,^\prime{{\cal B}_k}  \) in the strong topology.  More specifically, when \( {\cal B}_k'(U) \) is given the strong topology,  the space  \(  u({\,^\prime {\cal B}}_k(U) ) \) equipped with the subspace topology is topologically isomorphic to \(  \,^\prime{\cal{B}}_k(U)  \).  Thus in the case of differential forms \( {\cal B}_k(\R^n) \) question (i) has a negative, and (ii) has an affirmative answer.

\section{Topological Properties} 
\label{sec:preliminaries}

\begin{prop}\label{laundry}
	 \( \hat{\cal{B}}_k \) is an ultrabornological, bornological, barreled, \( (DF) \), Mackey, Hausdorff, and locally convex space.   \( \,^\prime{\cal B}_k  \) is barreled, \( (DF) \), Mackey, Hausdorff and locally convex.
\end{prop}
\begin{proof} 
	By definition, the topology on \( \hat{\cal{B}}_k \) is locally convex.   We showed that \( \hat{\cal{B}}_k \) is Hausdorff in  \cite{OC}.  
  According to K\:othe \cite{kothe}, p. 403, a locally convex space is a bornological \( (DF) \) space if and only if  it is the inductive limit of an increasing sequence of normed spaces.   It is \emph{ultrabornological} if it is the inductive limit of Banach spaces.       Therefore, \( \hat{\cal{B}}_k \) is an ultrabornological \( (DF) \)-space. 

	 Every inductive limit of metrizable convex spaces is a Mackey space (Robertson \cite{robertson} p. 82).  Therefore, \( \hat{\cal{B}}_k \) is a Mackey space. It is barreled according to Bourbaki \cite{bourbaki}    III 45, 19(a).

  The completion of any locally convex Hausdorff space is also locally convex and Hausdorff.  The completion of a barreled space is barreled by Schaefer \cite{scheafer}, p. 70 exercise 15  and  the completion of a \( (DF) \)-space is  \( (DF) \) by  \cite{scheafer} p.196 exercise 24(d). But the completion of a bornological space may not be bornological (Valdivia \cite{valdivia2})  
\end{proof}

\begin{thm}[Characterization 1]\label{mackey}
	The space of differential chains $\,^\prime {\cal B}_k$ is the completion of pointed chains ${\cal P}_k$ given the Mackey topology $\tau({\cal P}_k,\B_k)$.  
\end{thm}  

\begin{proof} 
Since \( ({\cal P}_k, \B_k) \) is a dual pair, the Mackey topology \( \tau({\cal P}_k,\B_k) \) is well-defined (Robertson \cite{robertson}).  This is  the finest locally convex topology on ${\cal P}_k$ such that the continuous dual ${\cal P}_k'$ is equal to $\B_k$. 

  	Let $\tau|_{{\cal P}_k}$ be the subspace topology on pointed chains ${\cal P}_k$ given the inclusion of ${\cal P}_k$ into $\hB_k$.    By (\ref{laundry}), the topology on $\hB_k$ is Mackey; explicitly it is the topology of uniform convergence on relatively \( \s({\cal B}_k, \hat{\cal{B}_k}) \)-compact sets where \( \s({\cal B}_k, \hat{\cal{B}_k}) \) is the weak topology on the dual pair \( ({\cal B}_k, \hat{\cal{B}_k}) \).  But then \( \t|_{{\cal P}_k} \)   is the topology of uniform convergence on relatively \( \s({\cal B}_k,  {\cal P}_k) \)-compact sets, which is the same as the Mackey topology \( \t \) on \( {\cal P}_k \).     Therefore, \( \t|_{{\cal P}_k} = \tau({\cal P}_k,\B_k) \).

\end{proof}

\section{Relation of Differential Chains to Currents} 
\label{sec:currents}

     The Fr\'echet 
topology \( F \) on \( {\cal B}  = {\cal B}_k\) is determined by the seminorms \( \|\omega\|_{C^r} =\sup_{J\in Q^r}\omega(J)\), where $Q^r$ is the image of the unit ball in $\hB_k^r$ via the inclusion\footnote{The authors show this inclusion is compact in a sequel.} $\hB_k^r\hookrightarrow\hB_k$.  

  \begin{lem}\label{lem:inclu}  
	\( ({\cal B}, \b({\cal B}, \,^\prime {\cal B})) = ({\cal B}, F) \). 
  \end{lem}      

\begin{proof} 
First, we show $F$ is coarser than $\b({\cal B}, \,^\prime {\cal B})$.  It is enough to show that the Fr\'echet seminorms are seminorms for $\beta$.  For every form $\o\in \B_k$ there exists a scalar $\lambda_{\o, r}$ such that $\|\lambda_{\o,r}\o\|_{C^r}\leq 1$.  In other words, ${Q^r}^0$, the polar of $Q^r$ in $\B_r$, is absorbent and hence $\|\cdot\|_{C^r}$ is a seminorm for $\beta$.  

On the other hand, $\b({\cal B}, \,^\prime {\cal B})$, as the strong dual topology of a $(DF)$ space, is also Fr\'echet, and hence by \cite{robertson} the two topologies are equal.
\end{proof}  
\begin{thm}\label{thm:pro}
  	The vector space \(    \,^\prime {\cal B}(\R^n)  \)  is a proper subspace of the vector space  \( {\cal B}'(\R^n) \).  

\end{thm}   

Remark:  \(  \,^\prime {\cal B} \) and \(  {\cal{B}} \) are barreled.  Thus semi-reflexive and reflexive are identical for both spaces.

\begin{proof}  Schwartz defines \( ({\cal B}) \) in \S 8 on page 55 of \cite{schwartz1} as the space of functions with all derivatives bounded on \( \R^n \), and endows it with   the Fr\'echet space topology, just as we have done.   He writes on page 56, ``\( ({\cal D}_{L^1}), (\dot{\cal B}),({\cal B}), \) ne sont pas r\'eflexifs.''     Suppose \( \,^\prime {\cal B}_k(\R^n) \) is reflexive.  By Lemma \ref{lem:inclu} the strong dual of  \( \,^\prime {\cal B}_k(\R^n) \) is \( ({\cal B}(\R^n), F) \). 
	 This implies that \( ({\cal B}(\R^n), F) \) is  reflexive, contradicting Schwartz.
\end{proof}

We immediately deduce:
  
\begin{thm}\label{topsubspace}
	The space $  \,^\prime {\cal B}_k$ carries the subspace topology of $\B_k'$, where $\B_k'$ is given the strong topology $\beta(\B_k', \B_k)$.  
\end{thm}

\begin{cor}[Characterization 2]\label{strong}
	The topology $\tau({\cal P}_k, \B_k)$ on ${\cal P}_k$ is the subspace topology on ${\cal P}_k$ considered as a subspace of $(\B_k', \beta(\B_k', \B_k))$.  
\end{cor}

\begin{remarks}
	We immediately see that ${\cal P}_k$ is not dense in $\B_k'$.  Compare this to the Banach-Alaoglu theorem, which implies ${\cal P}_k$ is weakly dense in $\B_k'$, whereas ${\cal P}_k$ is strongly dense in $  \,^\prime {\cal B}_k$.
 
In fact, Corollaries \ref{mackey} and \ref{strong} suggest a more general statement: let $E$ be an arbitrary locally convex topological vector space.  Elements of \( E \) will be our ``generalized forms.''  Let $P$ be the vector space generated by extremal points of open neighborhoods of the origin in $E'$ given the strong topology.  These will be our ``generalized pointed chains.''  Then $(P,E)$ forms a dual pair and so we may put the Mackey topology $\tau$ on $P$.  We may also put the subspace topology $\sigma$ on $P$, considered as a subspace of $E'$ with the strong topology.  We ask the following questions: under what conditions on $E$ will $\tau=\sigma$?  Under what conditions will $P$ be strongly dense in $E'$?  What happens when we replace $\B$ with $\mathcal{D}$, $\mathcal{E}$ or $\mathcal{S}$, the Schwartz space of forms rapidly decreasing at infinity?

\end{remarks}

\begin{thm}\label{antonyms}
	The space $\,^\prime {\cal B}_k(\R^n)$ is not nuclear, normable, metrizable, Montel, or reflexive.
\end{thm}
\begin{proof}
	 The fact that $\hB_k(\R^n)$ is not  reflexive follows from Theorem \ref{thm:pro}.  
	  It is well known that \( {\cal B}_k(\R^n) \) is not a normable space.   Therefore, \( \,^\prime {\cal B}_k(\R^n) \) is not normable.  
	   If \( E \) is metrizable and \( (DF) \), then \( E \) is normable (see p. 169 of Grothendieck \cite{AG}).    
	  Since \( \,^\prime {\cal B}_k(\R^n) \) is  a \( (DF) \) space, it is not metrizable.    If a nuclear space is complete, then it is semi-reflexive, that is, the space coincides with its second dual as a set of elements.    
  
Therefore, \( \,^\prime {\cal B}_k(\R^n) \) is not nuclear.    Since all Montel spaces are reflexive, we know that \( \hat{\cal{B}}_k(\R^n) \) is not Montel.         
\end{proof}
  
\section{Independent Characterization} 
\label{sec:independent_characterization}
We can describe our topology $\tau({\cal P}_k,\B_k)$ on ${\cal P}_k$ in another non-constructive manner for \( U \) open in \( \R^n \), this time without reference to the space $\B$.

\begin{thm}[Characterization 3]
	The topology $\tau({\cal P}_k,\B_k)$ is the finest locally convex topology \( \mu \) on ${\cal P}_k$ such that
  \begin{enumerate}
		\item  bounded mappings $({\cal P}_k,\mu) \to F$ are continuous whenever \( F \) is locally convex; 
		\item $K^0= \{(p;\a)\in {\cal P}_k : \|\a\|=1\}$ is bounded in \( ({\cal P}_k,\mu) \), where $\|\a\|$ is the mass norm of $\a \in \L_k(\R^n)$;
		\item $P_v: ({\cal P}_k,\mu)\rightarrow\overline{({\cal P}_k,\mu)}$ given by $P_v(p;\a):=\lim_{t\rightarrow 0}(p+v;\a/t)-(p;\a/t)$ is well-defined and bounded for all vectors $v\in \R^n$.  
	\end{enumerate}
\end{thm}

\begin{proof}
	A l.c.s. \( E \) is bornological if and only if bounded mappings \( S:E \to F \) are continuous whenever \( F \) is locally convex.	 Any subspace of a bornological space is bornological.  So by proposition \ref{laundry}, \( \t \) satisfies (1).  Properties (2) and (3) are established in Harrison \cite{diffcomp, OC}.

  Now suppose \( \mu \) satisfies (1)-(3).    Suppose \( \o \in ({\cal P}_k, \mu)' \).  Then                                
	\begin{align*}
	 \o P_v (p;\a) &= \o ( \lim_{t \to 0} (p + tv;\a/t) - (p;\a/t)) = \lim_{t \to 0} \o((p + tv;\a/t) - (p;\a/t) ) \\&= \lim_{t \to 0} \o(p + tv;\a/t) - \o(p;\a/t) = L_v \o (p;\a),
	\end{align*}  where \( L_v \) is the Lie derivative of \( \o \). 
	Since \( \o  \) is continuous and \( K^0 \) is bounded, it follows that \( \o(K^0) \) is bounded in \( \R \). (see \cite{bourbaki} III.11 Proposition 1(iii)).   Similarly, \( K^r  = P_v(K^{r-1}) \) is bounded implies \( \o(K^r) \) is bounded.  It follows that  \( \o \in \B_k \).  Hence \( ({\cal P}_k, \mu)' \subset \B_k \).  Now \( \mu \) is Mackey since it is bornological.   Since \(  ({\cal P}_k, \mu)' \subset ({\cal P}_k, \t)' \) and \( \t \) is also Mackey, we know that   \( \t \) is finer than \( \mu \) by the Mackey-Arens theorem.

\end{proof}

\begin{example}
	Let \( \|A\|_\natural = \varinjlim_r \|A\|_{B^r} \).  This is a norm on pointed chains (Harrison \cite{diffcomp}).   The Banach space  \( ({\cal P}, \natural) \) satisfies (1)-(3).  The topology \( \natural \) is strictly coarser than \( \t \) since  \( ({\cal P}, t)' = {\cal B} \) and \( ({\cal P}, \natural)'  \) is the space of differential forms with a uniform bound on \emph{all} directional derivatives.  
\end{example}


\end{document}